\documentclass[]{interact}

\usepackage{epstopdf}
\usepackage{subfigure}
\usepackage{graphicx}
\usepackage[numbers,sort&compress]{natbib}
\bibpunct[, ]{[}{]}{,}{n}{,}{,}

\theoremstyle{plain}
\newtheorem{theorem}{Theorem}[section]
\newtheorem{lemma}[theorem]{Lemma}
\newtheorem{corollary}[theorem]{Corollary}

\theoremstyle{definition}
\newtheorem{definition}[theorem]{Definition}

\newtheorem{algorithm}{Algorithm}[section]

\theoremstyle{remark}
\newtheorem{remark}{Remark}

\newcommand{\IR}{{\mathbb{R}}}

\newcommand{\BE}{\begin{equation}}
\newcommand{\EE}{\end{equation}}

\begin{document}


\title{Inexact Adaptive Cubic Regularization Algorithms on Riemannian Manifolds and Application}

	\author{
		\name{Z.~Y. Li\textsuperscript{a}\thanks{\textsuperscript{a}Email: Lzylllady@hotmail.com} and X.~M. Wang\textsuperscript{b}\thanks{\textsuperscript{b}Corresponding author, Email: xmwang2@gzu.edu.cn}}
		\affil{\textsuperscript{a,b}College of Mathematics and Statistics, Guizhou University, Guiyang 550025, P. R. China 
}
	}

\maketitle

\begin{abstract}
The adaptive cubic regularization algorithm employing the inexact gradient and Hessian is proposed on general Riemannian manifolds, together with the iteration complexity to get an approximate
second-order optimality under certain assumptions on accuracies about the inexact gradient and Hessian. The algorithm extends the inexact adaptive cubic regularization algorithm under true gradient in [Math. Program., 184(1-2): 35-70, 2020] to more general cases even in Euclidean settings.
As an application, the algorithm is applied to solve the joint diagonalization problem on the Stiefel manifold.
Numerical experiments illustrate that the algorithm performs better than the inexact trust-region algorithm in [Advances of the neural information processing systems, 31, 2018].
\end{abstract}

\begin{keywords}
Riemannian manifolds; retraction; inexact adaptive cubic regularization algorithm; iteration complexity; joint diagonalization 
\end{keywords}

\section{Introduction}

We consider the large-scale separable unconstrained optimization problem on general Riemannian manifolds:
\begin{equation} \label{eqn-1} 
	\min_{x\in \mathcal{M} }f(x)=\frac{1}{n}\sum_{i=1}^{n}f_{i}(x),
\end{equation}
where 
$\mathcal{M}$ is a Riemannian manifold, $n\gg 1$, and each $f_{i}:\mathcal{M} \rightarrow\mathbb{R}$ is continuously differentiable ($i:=1,2,\dots n$). Such problems frequently appear in machine learning and scientific computing,  where each $f_i$
is a loss (or misfit) function corresponding to $i-$th observation
(or measurement); see, e.g., \cite{Shalev-Shwartz2014,Roosta-Khorasani2013,Roosta-Khorasani2014}.
In such ``large-scale"  settings, 
since the evaluations of the gradient or the Hessian of $f$ can be computationally expensive, some inexact
techniques are used for approximating the first and the second derivatives, which particularly includes the random sampling technique. As a result, many popular first-order and second-order stochastic algorithms are proposed to solve problem \eqref{eqn-1}.
The stochastic gradient descent (SGD) algorithm is one of the first-order stochastic/inexact algorithms, which uses only the gradient information. The idea of the SGD comes from \cite{Robbins1951} and one can find some variants in \cite{Roux2012,Johnson2013,Bottou2018,Bonnabel2013,Sato2017} and the references therein (the Riemannian versions of the SGD can be found in \cite{Bonnabel2013,Sato2017}). 
The second-order stochastic/inexact algorithms always use the inexact Hessian information, which includes the popular sub-sampling Newton method, sub-sampling/inexact trust region algorithm, and the sub-sampling/inexact cubic regularization algorithm; see e.g,  \cite{Byrd2011,Erdogdu2015,Kohler2017,Yao2018,Xu2020,Kasai2018} (the Riemannian version of the inexact trust region algorithm is in \cite{Kasai2018}).
Particularly, we note the inexact trust-region algorithm and the inexact adaptive cubic regularization algorithms introduced in \cite{Xu2020} for solving problem \eqref{eqn-1}. Both algorithms employ the true gradient and the inexact Hessian (by sub-sampling) of $f$ and solve the corresponding sub-problems approximately every iteration. These two algorithms are respectively  shown to own iteration complexity $\mathcal{O} ( \max \{ \varepsilon _{g}^{-2}\varepsilon _{H}^{-1},\varepsilon _{H}^{-3}\}) $ and $\mathcal{O} ( \max \{ \varepsilon _{g}^{-2},\varepsilon _{H}^{-3}\}) $ of achieving the $(\varepsilon_g, \varepsilon_H)$-optimality (see definition of the $(\varepsilon_g, \varepsilon_H)$-optimality in Definition \ref{def4}).

It is known that the Riemannian geometry framework has some advantages in many applications, such as translating  some nonconvex (constrained) optimization problems in Euclidean space into convex (unconstrained) ones over a Riemannian manifold; see, e.g., \cite{Absil2008,WLWYSIAM2015,Ferrire2019,Bento2019M}. As a result, some classical numerical  methods  for solving optimization problems on the Euclidean space, such as Newton's method, BFGS algorithm, trust region method, gradient algorithm, subgradient algorithm,  etc., have been extended to the Riemannian manifold setting; see, e.g., \cite{Absil2007,Absil2008,Huang2015,Gabay1982,Ring2012,WLWYSIAM2015,WangOptim2018,Ferrire2019,Bento2019M}. As we have noted above, for solving large-scale problem  \eqref{eqn-1}, the stochastic gradient descent algorithm and its variants have been extended from Euclidean spaces to Riemannian manifolds in the literature; see, e.g., \cite{Bonnabel2013,Sato2017}.
Recently, the Riemannian inexact trust region algorithm is studied in \cite{Kasai2018}, which uses the inexact gradient and the inexact Hessian every iteration, and in addition the inexact solution of the trust region sub-problem. Noting that the algorithm employs the inexact gradient instead of the true gradient, it extends the corresponding one in \cite{Xu2020} to more general cases even in Euclidean settings.
Similarly, for achieving the $(\varepsilon_g, \varepsilon_H)$-optimality, the iteration complexity of the algorithm is also proven to be  $\mathcal{O} ( \max \{ \varepsilon _{g}^{-2}\varepsilon _{H}^{-1},\varepsilon _{H}^{-3}\}) $. At the same time, some numerical results are provided which illustrate the algorithm performs significantly better than state-of-the-art deterministic and stochastic algorithms in some applications.

Inspired by the prior works (in particular, the works in \cite{Xu2020} and \cite{Kasai2018}), we propose the Riemannian inexact adaptive cubic regularization algorithm for solving problem \eqref{eqn-1} on general Riemannian manifolds, which employs the inexact gradient and the inexact Hessian every iteration. The algorithm is proven to own the similar iteration complexity $\mathcal{O} ( \max \{ \varepsilon _{g}^{-2},\varepsilon _{H}^{-3}\}) $ for obtaining the $(\varepsilon_g, \varepsilon_H)$-optimality under certain conditions on accuracies about the inexact gradient and the inexact Hessian.
Particularly, iteration complexities of the Riemannian (deterministic) adaptive cubic regularization algorithm and the Riemannian inexact adaptive cubic regularization algorithm under the true gradient are established. As an application, the proposed algorithms are applied to solve the joint diagonalization problem on the Stiefel manifold. Numerical results indicate that inexact algorithms are more efficient than the deterministic algorithm at the same accuracy. Meanwhile, the inexact 
Riemannian adaptive cubic regularization algorithm outperforms the inexact Riemannian trust-region algorithm (in \cite[Algorithm 1]{Kasai2018}).

The paper is organized as follows. As usual,  some basic notions
and notation on Riemannian manifolds, together with some related  properties about the sub-sampling,  are introduced in the next section. The Riemannian inexact adaptive cubic regularization algorithm and its iteration complexity are presented in section 3, and the application to joint diagonalization on the Stiefel manifold is shown in the last section.

\section{Notation and Preliminaries} \label{sec2}
Notation and terminologies used in the present paper are standard;
the readers are referred to some textbooks for more details; see, e.g.,
\cite{Absil2008,Carmo1992,Udriste1994,Conn2000}.

Let $\mathcal{M}$ be a connected and complete $n$-dimensional Riemannian manifold. Let $x\in \mathcal{M}$,
and let ${T}_{x}\mathcal{M}$ stand for the tangent space at $x$ to $\mathcal{M}$. $T\mathcal{M}:=\cup_{x\in \mathcal{M}}T_{x}\mathcal{M}$ is called the tangent bundle on $\mathcal{M}$. We denote by
$\langle,\rangle_{x}$ the scalar product
 on $T_{x}\mathcal{M}$ with the associated norm $\|\cdot\|_{x}$, where the subscript $x$ is
sometimes omitted.
For $y\in\mathcal{M}$, let $\gamma:[0,1]\rightarrow \mathcal{M}$ be
a piecewise smooth curve joining $x$ to $y$. Then, the arc-length of
$\gamma$ is defined by $l(\gamma):=\int_{0}^{1}\|{\gamma}'(t)\|dt$,
while the Riemannian distance from $x$ to $y$ is defined by ${\rm
d}(x,y):=\inf_{\gamma}l(\gamma)$, where the infimum is taken over
all piecewise smooth curves $\gamma:[0,1]\rightarrow \mathcal{M}$ joining $x$
to $y$.

We use $\nabla$ to denote the Levi-Civita connection on $\mathcal{M}$. A vector field $V$ is said to be parallel along $\gamma$  if
$\nabla_{{\gamma}'}V=0$. In particular, for a smooth curve $\gamma$,
if ${\gamma}'$ is parallel along itself, then  $\gamma$ is called a
geodesic; thus, a smooth curve $\gamma$ is a geodesic if and only
if $\nabla_{{\gamma}'}{{\gamma}'}=0$. A geodesic
$\gamma:[0,1]\rightarrow M$ joining $x$ to $y$ is minimal if its
arc-length equals its Riemannian distance between $x$ and $y$. By
the Hopf-Rinow theorem \cite{Carmo1992}, $(\mathcal{M},{\rm d})$ is a complete
metric space, and there is at least one minimal geodesic joining $x$
to $y$ for any points $x$ and $y$.

Let $f:\mathcal{M}\rightarrow{\IR}$ be a twice continuously differentiable real-valued function defined on $\mathcal{M}$.
The Remannian gradient and Hessian of $f$ at $x\in \mathcal{M}$ are denoted by ${\rm grad}f(x)$  and ${\rm Hess} f(x)$, which are respectively defined as
$$
\langle{\rm grad}f(x),\xi\rangle=\xi(f)\quad\mbox{and}\quad {\rm Hess}f(x)(\xi,\eta)=\langle\nabla_{\xi}{\rm grad}f(x),\eta\rangle\quad\forall \xi\in T_{x}\mathcal{M}.
$$

On manifolds $\mathcal{M}$, we use the retraction defined as below to move in the direction of a tangent vector, which can be found in \cite[Definition 4.1.1 and Proposition 5.5.5]{Absil2008}.

 \begin{definition}[retraction and seconder-order retraction] \label{def1}
	A retraction on a manifold $\mathcal{M}$ is a smooth mapping $R: T\mathcal{M}\rightarrow \mathcal{M}$ with the following properties, where $R_{x}$ is the restriction of $R$ on $T_x\mathcal{M}$:
	
	{\rm (1)} $R_{x} (0_{x})=x$, where $0_{x}$ 	denotes the zero element of $T_{x}\mathcal{M}$.
	
	{\rm (2)} $DR_{x}( 0_{x})=\mathrm{Id}_{T_{x}\mathcal{M}}$, where $\mathrm{Id}_{T_{x}\mathcal{M}}$ is the identity mapping on $T_{x}\mathcal{M}$ with the canonical identification $T_{0_{x}}T_{x}\mathcal{M}\simeq  T_{x}\mathcal{M}$.
	
\noindent Moreover, $R$ is said to be a seconder-order retraction if it further satisfies
\[
\frac{D^2}{dt^2}R_{x}(t\eta) \mid_{t=0}=0\quad\forall x \in \mathcal{M}\mbox{ and }\forall \eta \in T_{x}\mathcal{M},
\]
where $\frac{D^2}{dt^2}\gamma$ denotes acceleration of the curve $\gamma$. 
\end{definition}

\begin{remark}\label{rem1}
(i) $\mathcal{M}$ admits a seconder-order retraction defined by the exponential mapping, and one can find more general retractions on matrix manifolds in \cite{Absil2008}.

(ii) In general, the Euclidean Hessian $\nabla ^2f\circ R_{x}(0_{x}) $ differs from the Riemannian Hessian $\mathrm{Hess}f(x)$, while    they are identical under second-order retractions (see\cite[Lemma 3.9]{Boumal2019}).
\end{remark}

Let $J:=\{1,2,\cdots,n\}$. We consider the large-scale minimization \eqref{eqn-1}: 
\[
\min_{x\in \mathcal{M}}f(x)=\frac{1}{n}\sum_{i=1}^{n}f_{i}(x),
\]
where $n\gg 1$, and for each $i\in J$, $f_{i}:\mathcal{M}\rightarrow\mathbb{R}$ is a twice continuously differentiable real-valued function. Note that $f$ may be non-convex, some authors refer to find a proximate $(\varepsilon_g,\varepsilon_H)$-optimality in practice, which is defined by \cite[Definition 1]{Xu2020} (see also \cite[Definition 2.1]{Kasai2018}). As usual, $\lambda_{\min}(A)$ stands the minimum eigenvalue of matrix $A$.  

\begin{definition}[$(\varepsilon_g,\varepsilon_H)$-optimality] \label{def4}
Let $\varepsilon _{g},\varepsilon _{H}\in(0,1)$. A point $x^*\in \mathcal{M}$ is said to be an $ (\varepsilon _{g},\varepsilon _{H} )$-optimality of (\ref{eqn-1}) if
\[
		 \| {\rm{grad}}f(x^*)\|\le \varepsilon_{g}\quad \mbox{and}\quad \lambda_{\min}({\rm{Hess}}f (x^*))\ge-\varepsilon _{H}.
\]
\end{definition}
In practice, we adopt the sub-sampled inexact gradient and Hessian. To proceed, let $S_g,S_H\subset J$ be the sample collections with or without replacement from $J$, and their cardinalities are denoted as $|S_{g}|$ and $|S_{H}|$, respectively.
The sub-sampled inexact gradient and Hessian are defined as
\begin{equation}\label{PGH}
G:=\frac{1}{|S_{g}|}\sum_{i\in S_{g}}\mathrm{grad}f_{i}\quad \mbox{and}\quad H:=\frac{1}{|S_{H}|}\sum_{i\in S_{H}}\mathrm{Hess}f_{i},
\end{equation}
respectively. 
The following lemma provides sufficient sample sizes of $S_g$  and $S_H$ to guarantee that the inexact $G$ and $H$ approximate ${\rm grad}f$ and ${\rm Hess}f$ in a probabilistic way, which is taken form \cite[Theorem 4.1]{Kasai2018}. Here 
we set 
	\[
	K^{\max}_{g}:=\max_{i\in J} \sup_{x\in \mathcal{M}}
	\| \mathrm{grad}f_{i} (x) \|\quad\mbox{and}\quad K^{\max}_{H}:=\max_{i\in J}\sup_{x\in \mathcal{M}}
	\| \mathrm{Hess}f_{i}(x) \|.
	\]	

\begin{lemma}\label{lem3.4}
Let $\delta, \delta_g, \delta_H\in(0,1)$ and let $R$ be a seconder-order retraction. Assume that the sampling is done uniformly at random to generate $S_g$ and $S_H$, and let $G$ and $H$ be defined by \eqref{PGH}.
If the sample sizes $|S_{g}|$ and $|S_{H}|$ satisfy
		\[
		|S_{g}|\ge \frac{32 (K_{g}^{\max})^2 (\log\frac{1}{\delta})+\frac{1}{4}}{\delta_{g}^2}\quad\mbox{and}\quad|S_{H}|\ge \frac{32(K_{H}^{\max})^2(\log\frac{1}{\delta}) +\frac{1}{4}}{\delta_{H}^2},
		\]
then the following estimates hold for any $x\in \mathcal{M}$ and any $\eta\in T_{x}\mathcal{M}$:
		\[
		\mathrm{Pr}(\|G(x)-\mathrm{grad}f(x)\|\le\delta_{g})\ge 1-\delta,
		\]
		\[
		\mathrm{Pr}(\|(H(x)-\nabla^2f\circ R_{x} (0_{x}))[\eta]\|\le \delta_{H}\|\eta\|)\ge 1-\delta.
		\]
\end{lemma}

\section{Inexact Riemannian adaptive cubic regularization algorithm}
\label{sec3}

In this section, we propose the following inexact Riemannian adaptive cubic regularization algorithm for solving problem \eqref{eqn-1}, which
employs the inexact Hessian and gradient at each iteration.

\begin{algorithm}(Inexact Riemannian adaptive cubic regularization)\label{SSRACR}
	
	\noindent {\textbf{Step\;0}}.  Choose $\varepsilon_g,\varepsilon_H,\rho_{TH}\in (0,1)$, and $\gamma\in(1,+\infty)$.
	
	\noindent {\textbf{Step\;1}}. Choose $x_{0}\in \mathcal{M}$ and $\sigma_0\in(0,+\infty)$, and set $k:=0$.
	
	\noindent {\textbf{Step\;2}}. Construct inexact gradient $G_k$ and approximate Hessian $H_k$ of $f$ at $x_k$, respectively.
	
	\noindent {\textbf{Step\;3}}. If $\| G_{k} \|\le \varepsilon_{g}$ and $\lambda_{\min}(H_{k})\ge-\varepsilon_{H}$, then return $x_{k}$.
	
	\noindent {\textbf{Step\;4}}. Solve the following sub-problem approximately:
	\begin{equation}
		\label{eqn-2}
		\eta_k\approx  { \underset{\eta  \in T_{x_k}\mathcal{M}}{{\arg\min} \, m_k(\eta)} :=\langle G_k,\eta \rangle+\frac{1}{2} \langle H_k [ \eta ],\eta   \rangle+\frac{1}{3} \sigma_k\| \eta  \|^3.}
	\end{equation}
	
	\noindent {\textbf{Step\;5}}. \label{step5}
	Set  $\rho_k:=\frac{f( x_k)-f\circ R_{x_{k}}( \eta_k )  }{-m_k( \eta _k ) }$.

	\noindent {\textbf{Step\;6}}. If $\rho_k\ge\rho_{TH}$, then set $x_{k+1}=R_{x_{k}}( \eta _k  )$ and $\sigma _{k+1}=\frac{\sigma_k}{\gamma }$; otherwise set  $x_{k+1}=x_{k}$ and $\sigma _{k+1}=\gamma \sigma_k$.
	
	\noindent {\textbf{Step\;7}}. Set $k:=k+1$ and go to Step 2.
\end{algorithm}

\begin{remark}\label{remarkm}
Algorithm \ref{SSRACR} is an extension version of the adaptive cubic regularization algorithm in Euclidean spaces. 
As we known, in Euclidean settings, the algorithm was proposed and extensively studied by Cartis et al. in \cite{Cartis2011-1,Cartis2011-2}. Recently, for solving the ``large-scale" separable problem \eqref{eqn-1}, Wu et al. proposed the adaptive cubic regularization algorithm with inexact Hessian (and true gradient) in \cite[Algorithm 2]{Xu2020} and the iteration complexity of the algorithm is proven to be $\mathcal{O} ( \max \{ \varepsilon _{g}^{-2},\varepsilon_{H}^{-3}\}) $.
\end{remark}

Let $\{x_k\}$ be a sequence generated by Algorithm \ref{SSRACR}. We make the following blanket assumptions throughout the paper:

\noindent {\bf (A1)}	There exists constant $L_H>0$ such that for all $k\in\mathbb{N}$$, \emph{}f\circ R$ satisfies
\[
			  | f\circ R_{x_k}( \eta_k )- f( x_k )- \langle \mathrm{grad}f( x_k),\eta_k \rangle-\frac{1}{2} \langle \nabla ^2f\circ R_{x_k}( 0_{x_k})[ \eta _k],\eta_k \rangle|
		     \le \frac{1}{2}L_{H}\| \eta_k \|^3.
\]

\noindent {\bf (A2)}	There exists constant $K_H>0$ 
such that 
  \[
     \| H_k \|:=\sup_{\eta \in T_{x_k}\mathcal{M}, \|\eta\|\le 1} \langle \eta ,H_k [ \eta ] \rangle\le K_H\quad \forall k\in\mathbb{N}.
   \]

\noindent {\bf (A3)}		There exist constants $\delta_g,\delta_H\in(0,1)$ such that 
			\begin{equation*}
\begin{array}{lll}
				&\mbox{\bf (A3)-a}\quad\| G_k-\mathrm{grad}f( x_k ) \|\le\delta_g\quad\forall k\in\mathbb{N},\\

				&\mbox{\bf(A3)-b}\quad \| ( H_k-\nabla^2f\circ R_{x_k}(0_{x_k}))[\eta_k] \|\le\delta_H \| \eta_k \|\quad\forall k\in\mathbb{N}.
				\end{array}
			\end{equation*}	

\begin{remark}\label{rem2}
(i) As shown in \cite[Appendix B]{Boumal2019}, Assumption (A1) is satisfied in the case when $\mathcal{M}$ is compact and $R$ is a second-order retraction.

(ii) In practice, according to Lemma \ref{lem3.4}, we shall construct $\{G_k\}$ and $\{H_k\}$ by sub-sampling to ensure Assumption (A3) in a probabilistic way.
\end{remark}

As stated in Step 4, we solve the sub-problem \eqref{eqn-2} approximately every iteration. The most popular conditions for the approximate solution in the literature are the Cauchy and Eigenpoint conditions; see, e.g., \cite{Cartis2011-1,Cartis2011-2,Xu2020}.
To ensure the convergence of the algorithm, we make the following similar assumptions on the approximate solutions $\{\eta_k\}$:

\noindent {\bf (A4)} $-m_k(\eta_k )\ge -m_k ( \eta_k^C )$, and $-m_k( \eta_k )\ge -m_k( \eta_k^E )\quad if  \enspace \lambda_{\min}(H_k)<0$\quad $\forall k\in\mathbb{N}$,	

\noindent where $\eta_{k}^C$ and $\eta_{k}^E$ are the approximate optimal solutions of \eqref{eqn-2} along the negative gradient and the negative curvature directions, respectively, and $\eta_{k}^E$ satisfies
	\[
	\langle \eta_k^E,H_k[\eta_k^E] \rangle\le\nu \lambda_{\min}( H_k) \| \eta_k^E \|^2< 0,\quad \nu \in (0,1);
	\]
see more details in  (\cite{Cartis2011-1,Cartis2011-2}).

Now, let $k\in\mathbb{N}$. Noting that the sub-problem (\ref{eqn-2}) of Algorithm \ref{SSRACR} is actually posed on Euclidean spaces, the follow properties of $\eta_{k}^C$ and $\eta_{k}^E$ are valid (see \cite[Lemmas 6,7]{Xu2020}).
	\begin{lemma}\label{lem3.5}
The following estimates for $\eta_{k}^C$ and $\eta_{k}^E$ hold:
		\begin{equation} \label{eqn-5}
			\resizebox{0.938\hsize}{!}{$
				-m_k(\eta_k^C)\ge \max
				\{ \frac{1}{12}||\eta_k^C||^2(\sqrt{K_H^2+4\sigma_k||G_k||}-K_H),
				\frac{||G_k||}{2\sqrt{3}}\min\{ \frac{||G_k||}{K_H},\sqrt{\frac{||G_k||}{\sigma_k}}\}
				\}, $}
		\end{equation}
		\begin{equation} \label{eqn-6}
			-m_k(\eta_k^E)\ge \frac{\nu|\lambda_{\min}(H_k)|}{6} \max \{ \|\eta_k^E\|^2,\frac{\nu^2|\lambda_{\min}(H_k)|^2}{\sigma_k^2}\},
		\end{equation}
		\begin{equation}	\label{eqn-7}
			\|\eta_{k}^C\|_{x_k}\ge \frac{1}{2\sigma_{k}}(\sqrt{K_H^2+4\sigma_{k}\|G_k\|}-K_H),
		\end{equation}
		\begin{equation}	\label{eqn-8}
			\|\eta_k^E\|\ge \frac{\nu|\lambda_{\min}(H_k)|}{\sigma_{k}}.
		\end{equation}
	\end{lemma}
	
	Lemma \ref{lme3.6} below estimates the sufficient decrease in the objective function. 
	
	\begin{lemma} \label{lme3.6}
		Assume  (A1) and (A3). Then we have
		\begin{equation} \label{eqn-9}
			f\circ R_{x_k}(\eta_k)-f(x_k)-m_k(\eta_k)\le \delta_g\|\eta_k\|+\frac{1}{2}\delta_H\|\eta_k\|^2    +(\frac{L_H}{2}-\frac{\sigma_k}{3})\|\eta_k\|^3.
		\end{equation}
	\end{lemma}
	\begin{proof}
		In view of \eqref{eqn-2}, we have that
		\[	
		\begin{split}
			&| f\circ R_{x_k}(\eta_k)-f(x_k)-m_k(\eta_k)+\frac{\sigma_k}{3}\|\eta_k\|^3|\\
			=&| f\circ R_{x_k}(\eta_k)-f(x_k)- \langle G_k,\eta_k  \rangle-\frac{1}{2} \langle H_k[\eta_k],\eta_k \rangle|  \\
			\le&| f\circ R_{x_k}(\eta_k)-f(x_k)-\langle {\rm gradf}(x_k),\eta_k \rangle-\frac{1}{2} \langle \nabla^2f\circ R_{x_k}(0_{x_k})[\eta_k],\eta_k \rangle| \\
			&+ | \langle G_k-{\rm gradf}f(x_k),\eta_k \rangle|+\frac{1}{2}| \langle
			( \nabla^2f\circ R_{x_k}(0_{x_k})-H_k ) [\eta_k],\eta_k  \rangle|  \\
			\le& \delta_{g}\|\eta_{k}\|
			+\frac{\delta_{H}}{2}\|\eta_{k}\|^2
			+\frac{L_H}{2}\|\eta_{k}\|^3,
		\end{split}
		\]
		where the second inequality is from Assumptions (A1) and (A3). Then, (\ref{eqn-9}) follows immediately.
	\end{proof}

	\begin{lemma} \label{lem3.7}
Assume (A1)-(A4). 
Assume further that $\sigma_k\ge 2L_H$ and  $\delta_g\le \frac{9\delta_H^2}{4\sigma_{k}}$. If one of the following conditions is satisfied, then parameter $\sigma_{k}$ decreases, that is, $\sigma_{k+1}=\frac{\sigma_{k}}{\gamma}\; (\gamma>1)$:
		\begin{enumerate} 
			\item [(a)]  \label{cona} 
			$\|G_k\|>\varepsilon_g$,\quad $\delta_{H}\le \min \{ \frac{1}{18},\frac{1-\rho_{TH}}{9} \} (\sqrt{K_H^2+4\sigma_k\varepsilon_g}-K_H)$.
			\item [(b)]  \label{b} 
			$\lambda_{\min}(H_k)<-\varepsilon_H$,\quad $\delta_{H}\le \min \{ \frac{1}{9},\frac{2(1-\rho_{TH})}{9} \} \nu \varepsilon_H$.
		\end{enumerate}
	\end{lemma}
	\begin{proof}
		In light of (\ref{eqn-9}) and  $\sigma_k\ge 2L_H$, we get that
		\begin{equation} \label{eqn-10}
			\begin{aligned}
				f\circ R_{x_k}(\eta_k)-f(x_k)-m_k(\eta_k)
				&\le ( \frac{L_H}{2}-\frac{\sigma_k}{3} )\|\eta_{k}\|^3+\frac{1}{2}\delta_H\|\eta_{k}\|^2+\delta_{g}\|\eta_{k}\| \\
				&\le -\frac{\sigma_{k}}{12}\|\eta_{k}\|^3+\frac{1}{2}\delta_{H}\|\eta_{k}\|^2+\delta_{g}
				\|\eta_{k}\|.  	
			\end{aligned}
		\end{equation}
We claim that the following implication holds:
		\begin{equation} 	\label{eqn-11}
			\mbox{[$\| \eta_k \|\ge \frac{9\delta_H}{\sigma_k}] \Longrightarrow [f\circ R_{x_k}(\eta_k)-f(x_k)-m_k(\eta_k)\le0]$}.
		\end{equation}
In fact, note that
\begin{equation}\label{factsq}
\mbox{$-\sigma_{k}t^2+6\delta_{H}t+12\delta_{g}\le0$ \quad for all $t\ge \frac{6\delta_{H}+\sqrt{36\delta_{H}^2+48\sigma_{k}\delta_{g}}}{2\sigma_{k}}=\frac{3\delta_H}{\sigma_k}( 1+\sqrt{1+\frac{4\sigma_k\delta_g}{3\delta_H^2}})$.}
\end{equation}
In view of $\delta_g\le \frac{9\delta_H^2}{4\sigma_{k}}$ (by assumption), we see that $\frac{4\sigma_k\delta_g}{3\delta_H^2}\le3$, and then have
$$\frac{3\delta_H}{\sigma_k}( 1+\sqrt{1+\frac{4\sigma_k\delta_g}{3\delta_H^2}})\le \frac{9\delta_H}{\sigma_k}.$$
If $\| \eta_k \|\ge \frac{9\delta_H}{\sigma_k}$, we get from \eqref{factsq} that $-\sigma_k \| \eta_k\|^2+6\delta_H \| \eta_k \|+12\delta_g\le 0$. Therefore, implication \eqref{eqn-11} is valid by \eqref{eqn-10}.

Below, we show that the following inequality holds under either condition (a) or (b):
		\begin{equation} \label{eqn-12}
			\rho_k\ge \rho_{TH}.
		\end{equation}
Granting this and 
Step 6 of Algorithm \ref{SSRACR}, we conclude that $\sigma_{k+1}=\frac{\sigma_{k}}{\gamma}$, completing the proof 
		
Assume that condition (a) is satisfied. Then, we have 
		\begin{equation} \label{eqn-14}
		\| \eta_k^C \|\ge \frac{1}{2\sigma_{k}}(\sqrt{K_H^2+4\sigma_{k}\varepsilon_g}-K_H)\ge \frac{9\delta_{H}}{\sigma_{k}},
		\end{equation}
where the first inequality is by (\ref{eqn-7}) and the first item of condition (a), and 
the last inequality thanks to the second item of condition (a). 
We claim that
		\begin{equation}
			\label{eqn-13}
			f\circ R_{x_k}(\eta_k)-f(x_k)-m_k(\eta_k)\le \frac{1}{2}\delta_H\| \eta_k^C \|^2+\delta_g \| \eta_k^C\|.
		\end{equation}
Indeed, in the case when $\| \eta_k \|\le\|\eta_k^C\|$, (\ref{eqn-13}) is evident from (\ref{eqn-10}). Otherwise, $\|\eta_k \|>\|\eta_k^C\|$. This, together with \eqref{eqn-14} and implication \eqref{eqn-11}, implies that $f\circ R_{x_k}(\eta_k)-f(x_k)-m_k(\eta_k)\le0$, and so claim \eqref{eqn-13} is trivial.
Using (\ref{eqn-14}), (\ref{eqn-13}) and  $\delta_g\le \frac{9\delta_H^2}{4\sigma_{k}}$, we estimate that
		\begin{equation}	\label{eqn-15}
			f\circ R_{x_k}(\eta_k)-f(x_k)-m_k(\eta_k)\le \frac{1}{2}\delta_H \| \eta_k^C \|^2+\frac{\delta_H}{4}\frac{9\delta_{H}}
			{\sigma_{k}}\| \eta_k^C \|\le \frac{3}{4}\delta_{H}\| \eta_k^C  \|^2.
		\end{equation}
Due to assumption (A4) and (\ref{eqn-5}), there holds that
		\[
		-m_k(\eta_{k})\ge -m_k(\eta_{k}^C)\ge \frac{1}{12}\| \eta_k^C \|^2(\sqrt{K_H^2+4\sigma_{k}\varepsilon_g}-K_H).
		\]
This together with (\ref{eqn-15}), implies that
		\[
		1-\rho_k=\frac{f\circ R_{x_k}(\eta_{k})-f(x_k)-m_k(\eta_{k})}{-m_k(\eta_k)}\le \frac{\frac{3}{4}\delta_{H}\| \eta_k^C  \|^2}{\frac{1}{12}\| \eta_k^C \|^2(\sqrt{K_H^2+4\sigma_{k}\varepsilon_g}-K_H)}\le 1-\rho_{TH},
		\]
where the second inequality thanks to the second item of condition (a), which shows (\ref{eqn-12}).
			
Assume that condition (b) is satisfied. Then, we get from (7) and condition (b) that
\begin{equation}	\label{eqn-17}
			\| \eta_k^E \|\ge \frac{\nu |\lambda_{\min}(H_k)|}{\sigma_{k}}
			\ge \frac{\nu \varepsilon_{H}}{\sigma_{k}}
			\ge \frac{9\delta_{H}}{\sigma_{k}},
		\end{equation}
Thus, we claim that
		\begin{equation}	\label{eqn-16}
			f\circ R_{x_k}(\eta_k)-f(x_k)-m_k(\eta_k)\le \frac{1}{2}\delta_H\| \eta_k^E \|^2+\delta_{g} \| \eta_k^E \|.
		\end{equation}
Indeed, if $\| \eta_k \|\le \| \eta_k^E \|$, claim (\ref{eqn-16}) is clear by (\ref{eqn-10}). Otherwise $\| \eta_k \|>\| \eta_k^E \|$. This, together with \eqref{eqn-17} and implication \eqref{eqn-11}, implies that $f\circ R_{x_k}(\eta_k)-f(x_k)-m_k(\eta_k)\le0$. Therefore, claim \eqref{eqn-16} holds trivially.
		In view of (\ref{eqn-17}), (\ref{eqn-16}) and  $\delta_g\le \frac{9\delta_H^2}{4\sigma_{k}}$, we get that
		\begin{equation} \label{eqn-18}
			f\circ R_{x_k}(\eta_k)-f(x_k)-m_k(\eta_k)\le \frac{1}{2}\delta_H \| \eta_k^E  \|^2+\frac{\delta_H}{4}\frac{9\delta_{H}}
			{\sigma_{k}} \| \eta_k^E \|\le \frac{3}{4}\delta_{H}\| \eta_k^E\|^2.
		\end{equation}
Moreover, in view of $\lambda_{\min}(H_k)<0$ by condition (b), we have from assumption (A4) and (\ref{eqn-6}) that
		\[
		-m_k(\eta_{k})\ge -m_k(\eta_{k}^E)\ge \frac{\nu |\lambda_{\min}(H_k)|}{6}
		\| \eta_k^E\|^2(\sqrt{K_H^2+4\sigma_{k}\varepsilon_g}-K_H).
		\]
Combining this and (\ref{eqn-18}), there holds that
		\[
		1-\rho_k=\frac{f\circ R_{x_k}(\eta_{k})-f(x_k)-m_k(\eta_{k})}{-m_k(\eta_k)}\le \frac{\frac{3}{4}\delta_{H} \| \eta_k^E \|^2}{\frac{\nu |\lambda_{\min}(H_k)|}{6} \| \eta_k^E \|^2}\le 1-\rho_{TH},
		\]
		where the second inequality thanks to the second item of condition (b), and then (\ref{eqn-12}) is ture. The proof is complete.
	\end{proof}

	In the following lemma, we estimate the upper bound for the cubic regularization parameters $\{\sigma_{k}\}$ before the algorithm terminates. Here all iterations satisfying (\ref{eqn-12}) are said to be successful and to be unsuccessful otherwise.
	\begin{lemma} \label{lem3.8}
		Assume (A1)-(A4). Assume further that Algorithm \ref{SSRACR} does not terminate at $N$-th iteration, and the parameters satisfy $\delta_g\le \frac{9\delta_H^2}{4\sigma_{k}}$ and
		\begin{equation}	\label{eqn-19}
			\delta_{H}\le \min\{ \min\{ \frac{1}{18},\frac{1-\rho_{TH}}{9} \}(\sqrt{K_H^2+4\sigma_k\varepsilon_g}-K_H),\min \{ \frac{1}{9},\frac{2(1-\rho_{TH})}{9} \}\nu\varepsilon _H \},	
		\end{equation}
		then
		\begin{equation}	\label{eqn-20}
			\sigma_k\le \max\{\sigma_0,2\gamma L_H\}\quad \forall k=1,2,\cdots,N.
		\end{equation}
	\end{lemma}
	\begin{proof}
By contradiction, we assume that the $k$-th iteration ($k\le N$) is the first unsuccessful iteration such that
		\begin{equation}  \label{eqn-21}
			\sigma_{k+1}=\gamma\sigma_{k}\ge 2\gamma L_H,
		\end{equation}
		which implies that $\sigma_k\ge 2L_H$. Since Algorithm \ref{SSRACR} does terminate at the $k$-th iteration, we have that either $\| G_k \|>\varepsilon_g$  or $\lambda_{\min}(H_k)<-\varepsilon_{H}$. Recalling that $\sigma_k\ge 2L_H$, $\delta_g\le \frac{9\delta_H^2}{4\sigma_{k}}$  and (\ref{eqn-19}), Lemma \ref{lem3.7} is applicable to showing $\sigma_{k+1}=\frac{\sigma_{k}}{\gamma}$, which contradicts with (\ref{eqn-21}). The proof is complete.
	\end{proof}

Without loss generality, we assume that $\sigma_0\le2\gamma L_H$. Moreover, from now on, 
let $\mathcal{N}_{succ}$  and $\mathcal{N}_{fail}$  stand the set of all the successful and unseccessful iterations of Algorithm \ref{SSRACR}, respectively. The number of successful (unsuccessful) iterations is denoted by $|\mathcal{N}_{succ}| $  $(| \mathcal{N}_{fail}|)$. The lemma below provides estimation of
the total number of successful iterations before the algorithm terminates.

	\begin{lemma} \label{lem3.9}
Assume that the assumptions made in Lemma \ref{lem3.8} are satisfied. Then the following estimate holds:
		\begin{equation}	\label{eqn-22}
			|\mathcal{N}_{succ}|\le \frac{f(x_0)-f_{\min}}{\rho_{TH}\kappa_{\sigma}}\max \{ \varepsilon_g^{-2}, \varepsilon_H^{-3}\}+1,
		\end{equation}
		where $f_{\min}$ is the minimum of $f$ over $\mathcal{M}$ and
\begin{equation}\label{kppa}
 \kappa_{\sigma}:=\{ \frac{\nu^3}{24\gamma ^2L_H^2},\frac{1}{2\sqrt{3}}\min \{ \frac{1}{K_H},\frac{1}{\sqrt{2\gamma L_H}}\} \}.
\end{equation}
	\end{lemma}

\begin{proof}
Noting that $\{ f(x_k)\} $ is monotonically decreasing, we have that
\begin{equation}\label{vds0}
\begin{array}{lll}
f(x_0)-f_{\min} &\ge \sum_{k\in \mathcal{N}_{succ}}\left(f(x_k)-f(x_{k+1})\right)\\
			& \ge \sum_{k\in \mathcal{N}_{succ}}\rho_{TH}(-m_k(\eta_k)),
\end{array}
\end{equation}
where the second inequality is by Steps 4-5. Assume that $k\in N_{succ}$ and Algorithm \ref{SSRACR} does not terminate at the $k$-th iteration. Then we have either $\| G_k \|>\varepsilon_g$ or $\lambda_{\min}(H_k)<-\varepsilon_{H}$, and
\begin{equation}\label{sigma-kl}
\sigma_k\le 2\gamma L_H.
\end{equation}
by lemma \ref{lem3.8}\eqref{eqn-20} (noting $\sigma_0\le2\gamma L_H$). In the case when $\| G_k \|>\varepsilon_g$, we get from (\ref{eqn-5}) that
\begin{equation}\label{mket}
\begin{array}{lll}
-m_k(\eta_k) &\ge \frac{\| G_k\|}{2\sqrt{3}}\min \{ \frac{\| G_k \|}{K_H},\sqrt{\frac{ \| G_k \|}{\sigma _k}} \}>\frac{\varepsilon_g}{2\sqrt{3} }\min \{ \frac{\varepsilon_g}{K_H},\sqrt{\frac{\varepsilon_g}{\sigma_k}} \} \\
			&>\frac{\varepsilon_g^2}{2\sqrt{3} }\min\{ \frac{1}{K_H},\sqrt{\frac{1}{\sigma _k}} \}>\frac{\varepsilon _g^2}{2\sqrt{3} }\min \{ \frac{1}{K_H},\sqrt{\frac{1}{2\gamma L_H}} \},
\end{array}
\end{equation}
where the third inequality is by $\varepsilon_g\in ( 0,1 )$, and the fourth inequality is because of \eqref{sigma-kl}.
Similarly, for the case when $\lambda_{\min}(H_k)<-\varepsilon_{H}$, we have from (\ref{eqn-6}) that
		\[
		-m_k(\eta_k)\ge \frac{\nu^3|\lambda_{\min}(H_k)|^3}{6\sigma _k^2}\ge \frac{\nu^3\varepsilon_H^3}{24\gamma ^2L_H^2}.
		\]
This, together with \eqref{mket}, yields 
		\begin{equation} \label{eqn-24}
			-m_k(\eta_k) \ge\min \{ \frac{\varepsilon _g^2}{2\sqrt{3} } \min \{ \frac{1}{K_H},\frac{1}{\sqrt{2\gamma L_H}} \},\frac{\nu^3\varepsilon _H^3}{24\gamma ^2L_H^2} \}
			\ge \kappa_{\sigma }\min \{ \varepsilon _g^2,\varepsilon _H^3 \},
		\end{equation}
where the last inequality thanks to \eqref{kppa}.	
Combing \eqref{vds0} and \eqref{eqn-24}, we get that
		\begin{align}
			f(x_0)-f_{\min} \ge(|\mathcal{N}_{succ}|-1)\rho_{TH}\kappa_{\sigma }\min \{ \varepsilon_g^2,\varepsilon _H^3\}.
		\end{align}
Thus, \eqref{eqn-22} holds, completing the proof.
	\end{proof}
	
	We can now prove the main theorem, which indicates the iteration complexity of Algorithm \ref{SSRACR}.

	\begin{theorem}[complexity of Algorithm \ref{SSRACR}] \label{the1}
Assume that the assumptions made in Lemma \ref{lem3.8} are satisfied. Then, Algorithm \ref{SSRACR} terminates after at most $\mathcal{O}( \max\{ \varepsilon _g^{-2},\varepsilon _H^{-3} \}  )$  iterations.
	\end{theorem}
	
\begin{proof}
Noting that the total number of iteration of Algorithm \ref{SSRACR} $N=|\mathcal{N}_{succ}|+|\mathcal{N}_{fail}|$, we have by definition of Algorithm \ref{SSRACR} that
		\begin{equation*}	\label{eqn-25}
			\sigma_{N}=\sigma_0\gamma ^{|\mathcal{N}_{fail}|-|\mathcal{N}_{succ}|}.
		\end{equation*}
Noting $\sigma_k\le 2\gamma L_H$ for each $k=1,2,\cdots N$ by Lemma \ref{lem3.8}\eqref{eqn-20}, it follows that
\begin{equation}\label{eqn-25}
|\mathcal{N}_{fail}|\le|\mathcal{N}_{succ}|+\log_{\gamma}\frac{2\gamma L_H}{\sigma _0}.	
\end{equation}
This, together with Lemma \ref{lem3.9}\eqref{eqn-22}, implies that
\begin{align}
			N=|\mathcal{N}_{succ}|+|\mathcal{N}_{fail}|&\le 2|\mathcal{N}_{succ}|+\log_{\gamma}{\frac{2\gamma L_H}{\sigma _0} } \nonumber \\
			&\le \frac{2(f(x_0)-f_{\min})}{\rho_{TH}\kappa_{\sigma}}\max \{ \varepsilon _g^{-2},\varepsilon _H^{-3} \}+1+\log_{\gamma}{\frac{2\gamma L_H}{\sigma _0} },\nonumber
		\end{align}
completing the proof.
	\end{proof}

As corollaries of Theorem \ref{the1}, we consider the (deterministic) Riemannian adaptive cubic regularization algorithm and the inexact Riemannian adaptive cubic regularization algorithm under the true gradient. The following corollary shows the iteration complexity of the (deterministic) Riemannian adaptive cubic regularization algorithm, which is immediate from Theorem \ref{the1} (noting that assumption (A3) is satisfied naturally).

\begin{corollary} \label{cor3.13}
Assume (A1)-(A2) and (A4). If Algorithm \ref{SSRACR} employs
\[
G_k:=\mathrm{grad} f(x)\quad and \quad H_k:=\mathrm{Hess} f(x) \quad \forall k\in \mathbb{N},
\]
then it terminates after at most $\mathcal{O}( \max\{ \varepsilon_g^{-2},\varepsilon_H^{-3}\}  )$ iterations.
	\end{corollary}


Following the same line of proof of Theorem \ref{the1}, we have the following corollary about the iteration complexity of the inexact Riemannian adaptive cubic regularization algorithm under the true gradient.

	\begin{corollary}  \label{cor3.12}
Assume (A1)-(A2), (A3)-b and (A4). Assume further that Algorithm \ref{SSRACR} employs
\[
	G_k:=\mathrm{grad} f(x) \quad \forall k\in \mathbb{N}
\]
and the parameter $\delta_H$ satisfies
		\[	
		\delta _H\le \min \{ \min \{ \frac{1}{12},\frac{1-\rho_{TH}}{6} \}(\sqrt{K_H^2+4\sigma_k\varepsilon _g}-K_H),\min \{ \frac{1}{6},\frac{1-\rho_{TH}}{3} \} \nu\varepsilon _H \}.
		\]
Then, the algorithm terminates after at most $\mathcal{O}( \max\{ \varepsilon_g^{-2},\varepsilon_H^{-3}\}  )$ iterations.
	\end{corollary}

\section{Application and numerical experiments}
	\label{sec4}

In this section, we shall apply Algorithm \ref{SSRACR} to solve 
the joint diagonalization (JD for short) problems on Stiefel manifolds, and provide some numerical experiments to compare their performances. 
  All numerical comparisons are implemented in MATLAB on a 3.20 GHz Intel Core machine with 8 GB RAM. 

We first introduce some notation and results about the Stiefel manifold and the JD. The following notation and results about the Stiefel manifold can be found in \cite{Absil2008}.
Let  $r,d\in \mathbb{N}$, 
and let  $A=(a_{ij})\in \mathbb{R}^{d\times r}$  (we always assume that $r\le d$).
The symmetric matrix of $A$ is denoted by $\mathrm{sym}(A)$, which is defined by $\mathrm{sym}(A):=\frac{A^T+A}{2}$. The Frobenius norm for $A$ is defined as
	\[
	\| \mathrm{A} \|_F^2:=\sum_{i=1}^{d} \sum_{j=1}^{r} a_{ij}^2.
	\]
We use ${\rm diag}(A)$ to denote the matrix formed by the diagonal elements of $A$, and then have $\|\mathrm{diag(A)}\|_F^2= \sum_{j=1}^{r} a_{jj}^2$ . The Stiefel manifold St$(r,d)$ is the set of all $r\times d$ orthogonal matrices, 
i.e.,
	\[
	\mathrm{St}(r,d):=\{X\in \mathbb{R}^{d\times r}:X^TX=I\},
	\]
where $I\in \mathbb{R}^{r\times r}$ is the identity matrix.
Let $U\in \mathrm{St}(r,d)$. The tangent space of St$(r,d)$  at $U$ is defined by
	\[
	T_U\mathrm{St}(r,d):={ \{\xi :\xi ^TU+U^T\xi =0\}}.
	\]
The Riemannian inner product on $T_U\mathrm{St}(r,d)$ is given by
	\[
	\langle \eta,\xi  \rangle_{U}:={\rm trace}(\eta^T\xi)\quad \forall \eta ,\xi \in T_U\mathrm{St}(r,d).
	\]
In our practice, we adopt the following retraction defined by
	$$R_{U}(\xi):={\rm qf}(U+\xi)\quad \forall U\in \mathrm{St}(r,d),  $$
where ${\rm qf}(U+\xi)$ means the orthogonal factor based on the QR decomposition of $U+\xi$. 

The following introduction about the blind source separation (BSS for short) and the JD can be found in the works in \cite{Theis2009,Kasai2018,Oja2000,Theis2006} and references there in.
Let $x(t)$ be a given $d$-dimensional stochastic process and $t$ the time index 
such that
	\[
	x(t)=As(t)
	\]
where $A=(a_{ij})\in\mathbb{R}^{d\times r}$ is a mixing matrix and $s(t)$ is the $r$-dimensional source vector, both of $A$ and $s(t)$ are unknown. The BSS problem tries to recover the mixing matrix $A$ and the source vector$s(t)$.
If we pose additional constraints on $s(t)$ (see, e.g., \cite{Theis2009}), the BSS problem can be transformed into the JD problem on the Stiefel manifold, which is stated as follows:
	\begin{equation} \label{eqn-26}
		\min_{{U}\in \mathrm{St}(r,d)}f_\mathrm{{ica}}{(U)}:=-\frac{1}{n}\sum_{i=1}^{n}\| \mathrm{diag} ({U^TC}_i{U}) \|_F^2,
	\end{equation}
where
$\{C_i\in\mathbb{R}^{d\times d}:i=1,2,\cdots,n\}$ are matrices estimated from the dada under some source conditions (see \cite{Theis2009,Theis2006} for more details). The following expressions of the (Riemannian) gradient and Hessian of $f{\rm{_{ica}}}$ are taken form \cite{Kasai2018}. Let ${U}\in \mathrm{St}(r,d)$.
The Riemannian gradient of $f\mathrm{_{ica}}(U)$ at $U$ is given by
	\[
	\mathrm{grad} f\mathrm{_{ica}}(U)=\mathrm{P}_U{\rm egrad} f\mathrm{_{ica}}(U)=\mathrm{P}_U(-\frac{1}{n}\sum_{i=1}^{n}4{C}_iU\mathrm{diag}({U^TC}_i{U})  ),
	\]
	where $\mathrm{egrad}f \mathrm{_{ica}}(U)$  is the Euclidean gradient of $ f\mathrm{_{ica}}(U)$ and the orthogonal projection $\mathrm{P}_U$  is
	\[
	\mathrm{P}_U(W) :={W-U{\rm sym}(U^TW) \quad \forall W}\in \mathbb{R}^{d\times r}.
	\]
The Riemannian Hessian of  $ f\mathrm{_{ica}}(U)$ is defined by
	\begin{align}
		\mathrm{Hess} f\mathrm{_{ica}}(U) [\xi]= &\mathrm{P}_U({\rm Degrad} f_\mathrm{{ica}} {(U)[\xi]}-\xi \mathrm{sym} ({U^T{\rm egrad}} f\mathrm{_{ica}}(U)) \nonumber \\
		&-U\mathrm{sym}({\xi^T{\rm egrad}}  f_\mathrm{{ica}}{(U)}  )-U\mathrm{sym} ({U^T{\rm Degrad}} f_\mathrm{{ica}}{(U)[\xi]}) \nonumber \quad \forall \xi \in T{_U}\mathrm{St}(r,d),
	\end{align}
	where
	\[
	\mathrm{Degrad} f_\mathrm{{ica}} {(U)[\xi]} =-\frac{1}{n}\sum_{i=1}^{n}4{C}_i({\xi {\rm diag}(U^TC} _i{U} )+{U{\rm diag} } ({\xi^TC} _i{U} )+{U{\rm diag}} ({U^TC} _i{\xi} )).
	\]

For simplicity, Algorithm \ref{SSRACR} with true gradient and true Hessain, with true gradient and inexact Hessain, and with inexact gradient and inexact Hessain are denoted by RACR, SRACR and SSRACR, respectively; and the inexact Riemannian trust-region algorithm in \cite[Algorithm 1]{Kasai2018} is denoted by SSRTR.
We shall apply RACR, SRACR, SSRACR, and SSRTR to solve the JD problem (\ref{eqn-26}), and provide some numerical comparisons with different $(n,d,r)$. In all numerical experiments, we
generate $\{{C}_i:i=1,2,\cdots, n\}$ randomly, and all related inexact gradients and inexact Hessains are constructed by sampling with sizes $|S_g|=\frac{n}{4}$  and $|S_H|=\frac{n}{40}$.
We set the parameters as $\sigma_{0}=0.001,\rho_{TH}=0.9$ and $\gamma=2$ in RACR, SRACR and SSRACR, and set $\triangle_0=1,\rho_{TH}=0.9$ and $\gamma=2$ in SSRTR (see \cite[Algorithm 1]{Kasai2018} for details about the papmeters). The stopping criterion of all executed algorithms is set as $\|G_k \|^2\le 0.001 $. 
	
The numerical results are listed in Table \ref{tab1}, where $i$ and $T$ denote the number of iterations and the CPUtime (in seconds), respectively. 
The observations from Table \ref{tab1} reveal that
SSRACR costs much less CPUtime than SRACR, RACR and SSRTR  at the same accuracy. 

	\begin{table}
	\tbl{Number of iterations and time for RACR, SRACR, SSRACR, and SSRTR.}
	{\begin{tabular}{lcccccccc} \toprule 
			& \multicolumn{2}{c}{RACR}  & \multicolumn{2}{c}{SRACR} &\multicolumn{2}{c}{SSRACR} &\multicolumn{2}{c}{SSRTR}\\ \cmidrule{2-9}
			$(n,d,r)$ & $i$ & $T$ & $i$ & $T$ & $i$ & $T$ & $i$ & $T$ \\ \midrule   
			(2015,5,5)& 317& 7.524& 788& 9.349& 220& 2.259& 398&4.042 \\
			(2015,10,10)& 1073& 48.144& 1247& 29.432& 718& 14.048& 1936&39.174 \\
			(2015,20,20)& 2209& 108.753& 2114& 87.839& 2570& 83.804& 5140&172.033 \\
			(2015,30,30)& 2693& 395.710& 2398& 177.815& 2603& 202.081& 12739&836.978\\
			(2015,43,43)& 2913& 1280.422& 3169& 736.996& 3563& 691.293& 12777&2622.408 \\
			(7200,43,43)& 2873& 5526.860& 3019& 2485.531& 2909& 2005.734& 12837&8501.568 \\
			(60000,43,43)& 298& 4448.468& 298& 2395.300& 303& 2118.032& 861&4579.865 \\
			\bottomrule
	\end{tabular}}
	\label{tab1}
\end{table}

\section*{Funding}

Research of the second author was supported in part by the National Natural Science Foundation of China (grant 12161017) and the Guizhou Provincial Science and Technology Projects (grant ZK[2022]110).

\end{document}